\documentclass{amsart}

\pdfoutput=1

\title{The Slice Spectral Sequence for the cyclic group of order $p$}
\author{Vigleik Angeltveit}
\address{Mathematical Sciences Institute \\
Australian National University \\
Canberra, ACT 0200 \\
Australia}

\usepackage{amsxtra}
\usepackage{amsfonts}
\usepackage[latin1]{inputenc}
\usepackage{graphicx}
\usepackage{amsmath,amssymb,latexsym,amsthm,mathrsfs}
\usepackage[all]{xy}
\usepackage{pstricks}
\usepackage{verbatim}
\usepackage{spectralsequences}
\usepackage{hyperref}

\newtheorem{theorem}{Theorem}[section]
\newtheorem{thm}[theorem]{Theorem}
\newtheorem{lemma}[theorem]{Lemma}
\newtheorem{corollary}[theorem]{Corollary}
\newtheorem{cor}[theorem]{Corollary}

\newtheorem{prop}[theorem]{Proposition}

\theoremstyle{definition}

\newtheorem{example}[theorem]{Example}

\makeatletter
\let\c@equation\c@theorem
\makeatother
\numberwithin{equation}{section}

\newcommand{\cA}{\mathcal{A}}              \newcommand{\cO}{\mathcal{O}}

     \newcommand{\bF}{\mathbb{F}}          
\newcommand{\bQ}{\mathbb{Q}} \newcommand{\bR}{\mathbb{R}}        \newcommand{\bZ}{\mathbb{Z}}

\newcommand{\sma}{\wedge} %smash product
\newcommand{\coker}{\textnormal{coker}}

\newcommand{\wh}{\widehat}

\newcommand{\cof}{\textnormal{cof}}
\newcommand{\pic}{\textnormal{Pic}}

\newcommand{\fib}{\textnormal{fib}}
\newcommand{\un}{\underline}
\newcommand{\Gsp}{Sp^G}
\newcommand{\tr}{\textnormal{tr}}
\newcommand{\im}{\textnormal{im}}

\begin{document}

\begin{abstract}
We describe the slice tower and slice spectral sequence for arbitrary suspensions of the Eilenberg-MacLane spectrum of an arbitrary Mackey functor for the cyclic group of prime order.
\end{abstract}

\maketitle

\section{Introduction}
The main goal of this paper is to describe the slice tower and slice spectral sequence for $\Sigma^V H\un{M}$ for an arbitrary virtual $C_p$-representation $V$ and an arbitrary $C_p$-Mackey functor $\un{M}$ when $p$ is prime. See Section \ref{s:slicetower} for a complete description of the slice or coslice tower of $\Sigma^V H\un{M}$, and see Section \ref{s:examples} for a complete description of the slice spectral sequence.

Some of this material is already in the literature: A discussion of $\Sigma^{n} H_{C_2} \un{\bZ}$ goes all the way back to Dugger \cite{Du06}, and there are of course calculations related to the cyclic group $C_8$ in Hill, Hopkins and Ravenel's solution to the Kervaire Invariant problem \cite{HHR}. Yarnall \cite{Ya17} has computed the slice spectral sequence for $\Sigma^n H_{C_{p^k}} \un{\bZ}$ for $p$ odd and $n>0$, and Hill, Hopkins and Ravenel \cite{HHR17} have computed the slice spectral sequence for $\Sigma^{m\lambda} H_{C_{p^k}} \un{\bZ}$ for $p$ odd and $m > 0$. Hill and Yarnall \cite{HiYa18} have an abstract description of the slices of a $C_p$-spectrum. Guillou and Yarnall \cite{GuYa} have a complete description of the slice spectral sequence for $\Sigma^n H_{C_2} \un{\bF_2}$ and $\Sigma^n H_{C_2 \times C_2} \un{\bF_2}$ for $n > 0$, and Carissa Slone \cite{Sl21} has a complete description of the slice spectral sequence for $\Sigma^n H_{C_2 \times C_2} \un{\bZ}$ for $n \in \bZ$. Multiple people, including Slone, have worked out the slice spectral sequence for $\Sigma^V H_{C_2} \un{\bF_2}$ and $\Sigma^V H_{C_2} \un{\bZ}$ in private.

The main contribution of this paper is a systematic treatment of the slice tower and slice spectral sequence of $\Sigma^V H_{C_p} \un{M}$ for any virtual $C_p$-representation $V$, any prime $p$, and any $C_p$-Mackey functor $\un{M}$. While a complete description of the slice or coslice tower is not all that complicated, a complete description of the slice spectral sequence splits into a number of different cases using a total of $18$ additional Mackey functors constructed from $\un{M}$. (Some of these are for $p=2$ only.) For specific $\un{M}$ the slice spectral sequence simplifies considerably, as many of these additional Mackey functors are either $0$ or isomorphic to each other.

\subsection{Conventions}
Throughout this paper $G$ will always denote a finite group, although it will soon become the cyclic group $C_p$. We will work in the category $\Gsp$ of genuine $G$-spectra as described in \cite[Section B.4]{HHR}, but none of our results depend on the particular point-set level foundations. For example, the reader who prefers to work with the infinity-category of genuine $G$-spectra instead may do so.

For a virtual $G$-representation $V=V_1 - V_2$ we will write $S^V$ for a cofibrant replacement of $\Omega^{V_2} \Sigma^{\infty} S^{V_1}$.

By the fiber or cofiber of a map we will always mean the homotopy fiber or homotopy cofiber. Any smash product, mapping space, or function spectrum should be interpreted in the derived sense.

\section{Grading conventions}
It is often convenient to grade equivariant stable homotopy groups on the Picard group of $\Gsp$ rather than on $RO(G)$. For $G=C_p$ for $p$ odd, the main result of \cite{An21A} (alternatively combine \cite[Theorem C]{AMR21} with \cite{Kr20}) says that this Picard group is isomorphic to $(\bZ/p)^{\times}/\{\pm 1\} \times \bZ^2$. The factor of $(\bZ/p)^{\times}/\{\pm 1\}$ comes from $\pic(A(C_p))$, where $A(C_p)$ is the Burnside ring of $C_p$, and $\pm a \in (\bZ/p)^{\times}/\{\pm 1\}$ is represented by the $C_p$-spectrum $S^{\lambda-\lambda(a)}$. Here $\lambda = \lambda(1)$, and $\lambda(a)$ is the $2$-dimensional real $C_p$-representation with the generator of $C_p$ acting by rotation by $\frac{a \cdot 2 \pi}{p}$.

By \cite[Prop 4.1]{HiYa18}, smashing with $S^{\lambda - \lambda(a)}$ preserves slices and hence the slice and coslice towers in Section \ref{s:slicecoslicetower} below.

This means that for the most part we can grade $C_p$-equivariant homotopy groups on just $\bZ\{1,\lambda\}$. Given a virtual representation $V = m \oplus \bigoplus\limits_{i=1}^{(p-1)/2} n_i \lambda(i)$ we can replace $V$ by $V' = m \oplus n \lambda$ with $n=\sum n_i$. If $V = \bigoplus n_i \lambda(i)$ where $\sum n_i=0$, smashing with $S^V$ has the effect of taking the box product of all homotopy Mackey functors with a twisted version $\un{\cA}^a$ of the Burnside Mackey functor. See \cite{An21A} or Example \ref{ex:twistedBurnside} below for details.

Similarly, we can replace the regular representation $\rho_{C_p}$ by $\rho_{C_p}' = 1 \oplus \frac{p-1}{2} \lambda$ at the cost of taking the box product of all homotopy Mackey functors with some $\un{\cA}^a$.

We use the same notation for $p=2$, and denote the sign representation $\sigma$ by $\frac{1}{2} \lambda$.

\section{The slice and coslice tower} \label{s:slicecoslicetower}
Recall from \cite{HHR} or \cite{Ull} that a $G$-spectrum $X$ has an associated slice tower
\[
 X \to \ldots \to P^{n+1} X \to P^n X \to \ldots,
\]
analogous to the Postnikov tower in non-equivariant homotopy theory. Here $P^n X$ is the Bousfield localization of $X$ with respect to a category $\tau_{n+1}$ which we briefly discuss below.

If we let $P_n X$ be the fiber of the map $X \to P^{n-1} X$ we get another tower
\[
 \ldots \to P_{n+1} X \to P_n X \to \ldots \to X
\]
analogous to the Whitehead tower in non-equivariant homotopy theory, which we will call the coslice tower of $X$.

The definition of the slice tower has changed slightly over time. In \cite{HHR}, Hill, Hopkins and Ravenel used the localizing subcategory $\tau_n^{orig}$ generated by $G_+ \sma_H S^{m\rho_H - \epsilon}$ where $\epsilon=0,1$ and $m\dim_\bR(\rho_H)-\epsilon \geq n$. Later Ullman \cite{Ull} realized that by not including the $\epsilon$ the slice tower has slightly better properties, so we let $\tau_n$ be the localizing subcategory generated by $G_+ \sma_H S^{m\rho_H}$ for $m\dim_{\bR}(\rho_H) \geq n$. (In \cite{Ull} $\tau_n^{orig}$ is denoted by $\tau_n$ and our $\tau_n$ is denoted by $\bar{\tau}_n$.) In the examples of interest in the solution to the Kervaire Invariant One problem we end up with the same slice tower and slice spectral sequence with the two versions of localizing subcategory, but we will stick with $\tau_n$ here.

The fiber $P^n_n X = \fib(P^n X \to P^{n-1} X)$, which is also the cofiber $\cof(P_{n+1} X \to P_n X)$, is the $n$-slice of $X$. By forgetting the $G$-action the slice tower becomes the Postnikov tower, and the coslice tower becomes the Whitehead tower.

From the slice tower we get a spectral sequence of $\bZ$-graded Mackey functors
\[
 E_2^{s,t} X = \un{\pi}_{t-s} P^t_t X \implies \un{\pi}_{t-s} X.
\]
With this grading convention the $d_r$ differential has the form
\[
 d_r : E_r^{s,t} \to E_r^{s+r, t+r-1}.
\]

The slice spectral sequence was a key ingredient in the Hill-Hopkins-Ravenel solution to the Kervaire Invariant One problem, and multiple recent papers including \cite{HHR17, GuYa} have studied the slice tower and slice spectral sequence in specific examples.

In addition to describing the slice tower of $X$ we might also wish to contemplate the coslice tower. In fact, we might wish to consider $P_m^n X = \fib(P^n X \to P^{m-1} X)$ for each $-\infty \leq m \leq n \leq \infty$. These all fit in the following diagram, which we will call the slice diagram of $X$:
\[ \xymatrix @C=1pc @R=1.5pc {
\ldots \ar[r] & P_{n+1} X \ar[r] \ar[d] & P_n X \ar[r] \ar[d] & P_{n-1} X \ar[r] \ar[d] & \ldots \ar[r] & X \ar[d] \\
& \vdots \ar[d] & \vdots \ar[d] & \vdots \ar[d] & & \vdots \ar[d] \\
& P^{n+1}_{n+1} X \ar[r] & P^{n+1}_n X \ar[r] \ar[d] & P^{n+1}_{n-1} X \ar[r] \ar[d] & \ldots \ar[r] & P^{n+1} X \ar[d] \\
& & P^n_n X \ar[r] & P^n_{n-1} X \ar[r] \ar[d] & \ldots \ar[r] & P^n X \ar[d] \\
& & & P^{n-1}_{n-1} X \ar[r] & \ldots \ar[r] & P^{n-1} X \ar[d] \\
& & & & & \vdots
} \]

By using the coslice tower instead of the slice tower we get another spectral sequence converging to $\un{\pi}_* X$, but because $P_m^n X$ can be described both as the fiber of $P^n X \to P_{m-1} X$ and as the cofiber of $P_{n+1} X \to P_m X$ the two spectral sequences are isomorphic. One reason for also considering the coslice tower is that sometimes the coslice tower is easier to describe than the slice tower.

We write $X \geq m$ if $X$ is in the localizing subcategory $\tau_m$, and we write $X \leq n$ if $\wh{S} \in \tau_k$ for $k > n$ implies $[\wh{S}, X] = 0$. We write $X \in [m,n]$ if $m \leq X \leq n$, which is equivalent to $X \simeq P_m^n X$.

There is a kind of duality between the the slice tower and coslice tower, given by Brown-Comentz duality. In \cite[Theorem 7.9]{Ull_thesis}, Ullman proves that for any $G$-spectrum $X$ the slice spectral sequence for its Brown-Comenetz dual $I_{\bQ/\bZ}X$ is the $\bQ/\bZ$-dual of the slice spectral sequence for $X$. This implies that by applying $I_{\bQ/\bZ}$ to the slice diagram, the slice tower for $X$ becomes the coslice tower for $I_{\bQ/\bZ} X$ and vice versa.

Brown-Comentz duality turns $\bZ$'s into $\bQ/\bZ$'s, and for that reason we include the constant Mackey functor on $\bQ/\bZ$ as one of our examples. It also turns a constant Mackey functor into a coconstant Mackey functor. If $A$ is torsion then $I_{\bQ/\bZ} \Sigma^V H \un{A} \simeq \Sigma^{-V} H \un{A}^*$, and $I_{\bQ/\bZ} \Sigma^V H\un{\bZ} \simeq \Sigma^{-V} H\un{\bQ/\bZ}^*$.

\section{Some Mackey functors} \label{s:Mackey}
For a general introduction to Mackey functors, see e.g.\ \cite{We00}. We picture a $C_p$-Mackey functor $\un{M}$ as a diagram
\[
 \un{M} = \vcenter{\hbox{\xymatrix{ \un{M}(C_p/C_p) \ar@/_/[d]_{R} \\ \un{M}(C_p/e) \ar@/_/[u]_{\tr} }} }
\]
Here $\un{M}(C_p/e)$ has a $C_p=\{1,t,\ldots,t^{p-1}\}$-action, which we omit from the diagram. For notational convenience we define $N = \sum\limits_{i=0}^{p-1} t^i$. We note a few properties:
\begin{enumerate}
\item For $x \in \un{M}(C_p/C_p)$, $R(x)$ lies in the $C_p$-fixed points of $\un{M}(C_p/e)$. Hence $R$ factors through the fixed points
\[
 \un{M}(C_p/e)^{C_p} = \ker \big( 1-t : \un{M}(C_p/e) \to \un{M}(C_p/e) \big).
\]

\item For $x \in \un{M}(C_p/e)$, $\tr(x) = \tr(t x)$. Hence $\tr$ factors through the orbits
\[
 \un{M}(C_p/e)_{C_p} = \coker \big( 1-t : \un{M}(C_p/e) \to \un{M}(C_p/e) \big).
\]

\item For $x \in \un{M}(C_p/e)$, $R(\tr(x)) = N x$.
\end{enumerate}

\begin{example}
Let $A$ be an abelian group. Then
\[
 \wh{A} = \vcenter{\hbox{\xymatrix{ A \ar@/_/[d] \\ 0 \ar@/_/[u] }} }
\]
is a Mackey functor.
\end{example}

\begin{example} \label{ex:fixedandorbit}
Let $A$ be a $\bZ[C_p]$-module. Then we have a ``fixed point'' Mackey functor
\[
 F(A) = \vcenter{\hbox{\xymatrix{ A^{C_p} \ar@/_/[d]_{incl} \\ A \ar@/_/[u]_{N} }} }
\]
The restriction map $R$ is given by inclusion of fixed points, and $\tr(x) = N x$.

We also have an ``orbit'' Mackey functor
\[
 \cO(A) = \vcenter{\hbox{\xymatrix{ A_{C_p} \ar@/_/[d]_{N} \\ A \ar@/_/[u]_{\pi} }} }
\]
In this case the transfer map $\tr$ is the quotient map $x \mapsto [x]$, and $R([x]) = N x$ (which is independent of choice of representative $x$ of $[x]$).

We have a map $N_A : \cO(A) \to F(A)$ given by
\[ \xymatrix{
 A_{C_p} \ar@/_/[d] \ar[r]^-{N} & A^{C_p} \ar@/_/[d] \\
 A \ar[r]^-1 \ar@/_/[u] & A \ar@/_/[u]
} \]
Here $N$ is once again the map $[x] \mapsto Nx$. We can then consider the Mackey functors
\[
 \wh{\ker(N_A)} = \vcenter{\hbox{\xymatrix{ \ker \big( N : A_{C_p} \to A^{C_p} \big) \ar@/_/[d] \\ 0 \ar@/_/[u] }} }
 \qquad
 \wh{\coker(N_A)} =\vcenter{\hbox{\xymatrix{ \coker \big( N : A_{C_p} \to A^{C_p} \big) \ar@/_/[d] \\ 0 \ar@/_/[u] }} }
\]
as well as
\[
 \un{\im(N_A)} = \vcenter{\hbox{\xymatrix{ \im \big( N : A_{C_p} \to A^{C_p} \big) \ar@/_/[d]_{incl} \\ A \ar@/_/[u]_{N} }} }
\]
Here $\tr$ is given by $N$ and $R$ is given by inclusion.

For an arbitrary $C_p$-Mackey functor $\un{M}$, we write $F(\un{M})$ for $F(\un{M}(C_p/e))$, and similarly for $\cO(\un{M})$, $\wh{\ker(N_{\un{M}})}$, $\wh{\coker(N_{\un{M}})}$ and $\un{\im(N_{M})}$. Hence all of these Mackey functors depend only on $\un{M}(C_p/e)$, and not on $\un{M}(C_p/C_p)$.
\end{example}

\begin{example}
If the $C_p$-action on $A$ is trivial then $F(A) = \un{A}$ is the constant Mackey functor (with $R=1$ and $\tr = p$) and $\cO(A) = \un{A}^*$ is the coconstant Mackey functor (with $R=p$ and $\tr=1$). The map $N$ is multiplication by $p$, so in this case $\wh{\ker(N_A)} = \wh{A[p]}$, where $A[p] = \{ x \in A \quad | \quad px = 0\}$. Similarly, $\wh{\coker(N_A)} = \wh{A/p}$. Finally,
\[
 \un{\im(N_A)} = \vcenter{\hbox{\xymatrix{ pA \ar@/_/[d]_{1} \\ A \ar@/_/[u]_{p} }} }
\]
In this case $R$ is the inclusion map and $\tr$ is multiplication by $p$.

To be even more specific, we consider the three abelian groups $\bZ/p^n$, $\bZ$, and $\bQ/\bZ$. For $\bZ/p^n$ we get the following Mackey functors:
\[
\wh{\ker(N_{\bZ/p^n})} = \wh{\bZ/p} \qquad \wh{\coker(N_{\bZ/p^n})} = \wh{\bZ/p} \qquad \un{\im(N_{\bZ/p^n})} = \vcenter{\hbox{\xymatrix{ \bZ/p^{n-1} \ar@/_/[d]_p \\ \bZ/p^n \ar@/_/[u]_1 }} }
\]
For $\bZ$ we get the following Mackey functors:
\[
 \wh{\ker(N_{\bZ})} = 0 \qquad \wh{\coker(N_{\bZ})} = \wh{\bZ/p} \qquad \un{\im(N_{\bZ})} = \un{\bZ}^*.
\]
Finally, for $\bQ/\bZ$ we get the following Mackey functors:
\[
 \wh{\ker(N_{\bQ/\bZ})} = \wh{\bZ/p} \qquad \wh{\coker(N_{\bQ/\bZ})} = 0 \qquad \un{\im(N_{\bQ/\bZ})} = \un{\bQ/\bZ}.
\]

\end{example}

\begin{example} \label{ex:SESA}
For an arbitrary $C_p$-Mackey functor $\un{M}$, $R$ need not be injective. Let
\[
 \wh{\ker(R)} = \vcenter{\hbox{\xymatrix{ \ker \big( R : \un{M}(C_p/C_p) \to \un{M}(C_p/e) \big) \ar@/_/[d] \\ 0 \ar@/_/[u] }} }
\]
and let
\[
 \un{\im(R)} = \vcenter{\hbox{\xymatrix{ \im(R) \ar@/_/[d]_{\textnormal{incl}} \\ \un{M}(C_p/e) \ar@/_/[u]_N }} }
\]

Then we have a short exact sequence of Mackey functors
\[
 0 \to \wh{\ker(R)} \to \un{M} \to \un{\im(R)} \to 0,
\]
and $R$ is injective for the Mackey functor $\un{\im(R)}$.
%If $R$ was already injective for $\un{M}$, $\un{\im(R)} \cong \un{M}$.

We can compare $\un{\im(R)}$ to $F(\un{M})$ by another short exact sequence of Mackey functors:
\[
 0 \to \un{\im(R)} \to F(\un{M}) \to \wh{\coker(R)} \to 0
\]
Here $\coker(R)$ denotes the cokernel of $R : \un{M}(C_p/C_p) \to \un{M}(C_p/e)^{C_p}$. (Not of $R : \un{M}(C_p/C_p) \to \un{M}(C_p/e)$.)
\end{example}

\begin{example} \label{ex:SESB}
For an arbitrary $C_p$-Mackey functor $\un{M}$, $\tr$ need not be surjectve. Let
\[
 \wh{\coker(\tr)} = \vcenter{\hbox{\xymatrix{ \coker \big( \tr : \un{M}(C_p/e) \to \un{M}(C_p/C_p) \big) \ar@/_/[d] \\ 0 \ar@/_/[u] }} }
\]
and let
\[
 \un{\im(\tr)} = \vcenter{\hbox{\xymatrix{ \im(\tr) \ar@/_/[d]_{R} \\ \un{M}(C_p/e) \ar@/_/[u]_{\tr} }} }
\]

Then we have a short exact sequence of Mackey functors
\[
 0 \to \un{\im(\tr)} \to \un{M} \to \wh{\coker(\tr)} \to 0,
\]
and $\tr$ is surjective for the Mackey functor $\un{\im(\tr)}$.

We can compare $\un{\im(\tr)}$ to $\cO(\un{M})$ by another short exact sequence of Mackey functors:
\[
 0 \to \wh{\ker(\tr)} \to \cO(\un{M}) \to \un{\im(\tr)} \to 0.
\]
Here $\ker(\tr)$ denotes the kernel of $\tr : \un{M}(C_p/e)_{C_p} \to \un{M}(C_p/C_p)$. (Not of $\tr : \un{M}(C_p/e) \to \un{M}(C_p/C_p)$.)
\end{example}

\begin{example} \label{ex:p=2MFs}
At $p=2$ there are a few more Mackey functors to consider. For a $\bZ[C_2]$-module $A$, let $A_- = A \otimes \bZ_-$ denote $A$ with $t$ acting by the negative of the original action.

Given a $C_2$-Mackey functor $\un{M}$, we define a new Mackey functor $\underline{\ker(\tr)}_-$ as follows:
\[
 \underline{\ker(\tr)}_- = \vcenter{\hbox{\xymatrix{ \ker \big( \tr : \un{M}(C_2/e) \to \un{M}(C_2/C_2) \big) \ar@/_/[d]_{R} \\ \un{M}(C_2/e)_- \ar@/_/[u]_{\tr} }} }
\]
Here the restriction map is given by inclusion and the transfer map is given by $x \mapsto x - tx$. We note that if $x \in \ker(\tr)$ then $x + tx = 0$, so $tx = -x$.

Similarly, we define the Mackey functor $\underline{\ker(N_M)}_-$ as
\[
 \underline{\ker(N_M)}_- = \vcenter{\hbox{\xymatrix{ \ker \big( N : \un{M}(C_2/e) \to \un{M}(C_2/e) \big) \ar@/_/[d]_{R} \\ \un{M}(C_2/e)_- \ar@/_/[u]_{\tr} }} }
\]

We also define a Mackey functor $\un{\coker(R)}_-$ as
\[
 \underline{\coker(R)}_- = \vcenter{\hbox{\xymatrix{ \coker \big( R : \un{M}(C_2/C_2) \to \un{M}(C_2/e) \big) \ar@/_/[d]_{R} \\ \un{M}(C_2/e)_- \ar@/_/[u]_{\tr} }} }
\]
and a Mackey functor $\un{\coker(N_M)}_-$ as
\[
 \underline{\coker(N_M)}_- = \vcenter{\hbox{\xymatrix{ \coker \big( N : \un{M}(C_2/e) \to \un{M}(C_2/e) \big) \ar@/_/[d]_{R} \\ \un{M}(C_2/e)_- \ar@/_/[u]_{\tr} }} }
\]
\end{example}

\begin{example} \label{ex:twistedBurnside}
The Burnside Mackey functor of $C_p$ is
\[
 \un{\cA} = \vcenter{\hbox{\xymatrix{ \bZ \oplus \bZ \ar@/_/[d]_{\left(\begin{smallmatrix} 1 & p \end{smallmatrix}\right)} \\ \bZ \ar@/_/[u]_{\left( \begin{smallmatrix} 0 \\ 1 \end{smallmatrix} \right)} }} }
\]
As explained in \cite{An21A}, the Picard group of the Burnside ring of $C_p$ can be identified with the Picard group of the category of $C_p$-Mackey functors. It comes out to be $(\bZ/p)^\times/\{\pm 1\}$, and $a$ is represented by the Mackey functor
\[
 \un{\cA}^a = \vcenter{\hbox{\xymatrix{ \bZ \oplus \bZ \ar@/_/[d]_{\left(\begin{smallmatrix} a & p \end{smallmatrix}\right)} \\ \bZ \ar@/_/[u]_{\left( \begin{smallmatrix} 0 \\ 1 \end{smallmatrix} \right)} }} }
\]
All the Mackey functors in this paper that are constructed from $\un{M}$, except for $\un{M}$ itself, are invariant under $\un{\cA}^a \Box -$. For example, if we consider the first short exact sequence in Example \ref{ex:SESA} applied to $\un{M} = \un{\cA}$ and apply $\un{\cA}^a \Box -$ we get the short exact sequence
\[
 0 \to \wh{\bZ} \to \un{\cA}^a \to \un{\bZ} \to 0.
\]
Here $\wh{\bZ}$ and $\un{\bZ}$ do not depend on $a$, in the sense that $\un{\cA}^a \Box \wh{\bZ} \cong \wh{\bZ}$ and $\un{\cA}^a \Box \un{\bZ} \cong \un{\bZ}$.

Smashing with $S^{\lambda - \lambda(a)}$ has the effect of twisting all homotopy Mackey functors by $\un{\cA}^a \Box -$, so if $V = m \oplus \bigoplus n_i \lambda(i)$ and $V' = m \oplus n$ for $n = \sum n_i$ then smashing with $S^{V'-V}$ has this effect for some $a$. This $a$ can be explicitly computed as $a = \prod i^{n_i}$ in $(\bZ/p^\times)/\{\pm 1\}$.
\end{example}

\section{The homology of a point}
In this section we compute the homology groups, or rather the homology Mackey functors, of a point.

The calculations for $C_p$ for $p=2$ and for $p$ odd are almost identical, except for the additional complication that when $p=2$ we have to allow for $V=m+n\lambda$ with $n \in \frac{1}{2} \bZ$.

\begin{prop} \label{p:homologyofpoint}
Let $p$ be any prime, let $V = m \oplus \bigoplus n_i \lambda_i$, and let $n = \sum n_i$. Then $\un{H}_{m+k}(S^V; \un{M})$ is as follows:
\begin{enumerate}
 \item If $n=0$ then $\un{H}_{m+k}(S^V; \un{M})$ is nonzero for $k=0$ only and
 \[
  \un{H}_m(S^V; \un{M}) = \un{\cA}^a \Box \un{M} \qquad a = \big( \prod i^{n_i} \big)^{-1}
 \]
 Here $\un{\cA}^a$ is defined in Example \ref{ex:twistedBurnside}, and $a$ is computed in $(\bZ/p)^\times/\{\pm 1\}$.
%  \item If $n=0$ then $\un{H}_{m}(S^V; \un{M}) = \un{M}^a$, where $\un{M}^a = \un{\cA}^a \Box \un{M}$ for $a = \big( \prod i^{n_i} \big)^{-1}$ and $\un{\cA}^a$ is defined in Example \ref{ex:twistedBurnside}, and $\un{H}_{m+k}(S^V; \un{M}) = 0$ for $k \neq 0$.

 \item If $n > 0$ then $\un{H}_{m+k}(S^V; \un{M})$ is nonzero for $k \in [0,2n]$ only:
 \begin{itemize}
  \item $k=0$: $\wh{\coker(\tr)}$
  \item $k=1<2n$: $\wh{\ker(\tr)}$
  \item $k=1=2n$: $\un{\ker(\tr)}_-$ ($p=2$ only)
  \item $k=2i+1<2n$ for $i \geq 1$: $\wh{\ker(N_{\un{M}})}$
  \item $k=2i+1=2n$ for $i \geq 1$: $\un{\ker(N_M)}_-$ ($p=2$ only)
  \item $k=2i<2n$ for $i \geq 1$: $\wh{\coker(N_{\un{M}})}$
  \item $k=2i=2n$ for $i \geq 1$: $F(\un{M})$
 \end{itemize}

 \item If $n < 0$ then $\un{H}_{m+k}(S^V; \un{M})$ is nozero for $k \in [2n,0]$ only:
 \begin{itemize}
  \item $k=0$: $\wh{\ker(R)}$
  \item $k=-1 > 2n$: $\wh{\coker(R)}$
  \item $k=-1=2n$: $\un{\coker(R)}_-$ ($p=2$ only)
  \item $k=2i-1>2n$ for $i \leq -1$: $\wh{\coker(N_{\un{M}})}$
  \item $k=2i-1=2n$ for $i \leq -1$: $\un{\coker(N_M)}_-$ ($p=2$ only)
  \item $k=2i>2n$ for $i \leq -1$: $\wh{\ker(N_{\un{M}})}$
  \item $k=2i=2n$: for $i \leq -1$: $\cO(\un{M})$
 \end{itemize}

\end{enumerate}

\end{prop}

Except in the definition of $\un{\ker(\tr)}_-$, $\un{\ker(N_M)}_-$, $\un{\coker(R)}_-$ and $\un{\coker(N_M)}_-$ we view $R$ as a map $\un{M}(C_p/C_p) \to \un{M}(C_p/e)^{C_p}$, $\tr$ as a map $\un{M}(C_p/e)_{C_p} \to \un{M}(C_p/C_p)$, and $N_{\un{M}}$ as a map $\un{M}(C_p/e)_{C_p} \to \un{M}(C_p/e)^{C_p}$.

\begin{proof}
Part (1) follows from the discussion in Example \ref{ex:twistedBurnside} above.

For Part (2) we first replace $V$ with $V' = m \oplus n \lambda$. Then we compute the homology Mackey functors from a $G$-CW structure on $S^{V'}$. We can put a $G$-CW structure on $S^{V'}$ with a $C_p/C_p$-cell in dimension $m$ and a $C_p/e$-cell in dimension $m+k$ for $1 \leq k \leq 2n$. The homology Mackey functors are given by the homology of the following chain complex (drawn here for $n=2$):
\[ \xymatrix @C=1.4pc {
 \un{M}(C_p/C_p) \ar@/_/[d]_R & \un{M}(C_p/e) \ar[l]_-{\tr} \ar@/_/[d]_\Delta & \un{M}(C_p/e) \ar[l]_-{1-t} \ar@/_/[d]_\Delta & \un{M}(C_p/e) \ar[l]_-{N} \ar@/_/[d]_\Delta & \un{M}(C_p/e) \ar[l]_-{1-t} \ar@/_/[d]_\Delta \\ \un{M}(C_p/e) \ar@/_/[u]_{\tr} & \oplus_p \un{M}(C_p/e) \ar[l]_-{\nabla'} \ar@/_/[u]_\nabla & \oplus_p \un{M}(C_p/e) \ar[l]_-{1-sh} \ar@/_/[u]_\nabla & \oplus_p \un{M}(C_p/e) \ar[l]_-{N'} \ar@/_/[u]_\nabla & \oplus_p \un{M}(C_p/e) \ar[l]_-{1-sh} \ar@/_/[u]_\nabla
} \]
Here $N=\sum\limits_{i=0}^{p-1} t^i$ as before. The $C_p$-action on $\oplus_p \un{M}
(C_p/e)$ is given by cyclically permuting the summands. In this case we denote the $t$-action by $sh$, and define $N' = \sum\limits_{i=0}^{p-1} sh^i$. The map $\nabla'$ is given by $(x_0,\ldots,x_{p-1}) \mapsto \sum\limits_{i=0}^{p-1} t^i x_i$, while $\Delta$ and $\nabla$ are the usual diagonal and fold map, respectively. We note that to get these formulas we have to make a choice of identification $C_p/e \times C_p/e \cong \coprod\limits_p C_p/e$, and making a different choice leads to slightly different formulas. Taking homology gives Part (2).

Part (3) is similar, starting with a $G$-CW structure on $S^{V'}$ with a $C_p/C_p$-cell in dimension $m$ and a $C_p/e$-cell in dimension $m+k$ for $2n \leq k \leq -1$.
\end{proof}

\begin{cor} \label{c:OversusF}
For any $C_p$-Mackey functor $\un{M}$, $\Sigma^{\lambda} H \cO(\un{M}) \simeq \Sigma^2 F(\un{M})$
\end{cor}

In particular, for any abelian group $A$, $\Sigma^{\lambda} H\un{A}^* \simeq \Sigma^2 H\un{A}$.

\begin{proof}
With $\un{M} = \cO(A)$ and $n=1$ the above result says that $\un{\pi}_k \Sigma^{\lambda} H \cO(A) \cong F(A)$ for $k=2$ and trivial for $k \neq 2$ since both $\wh{\coker(\tr)}=0$ and $\wh{\ker(\tr)} = 0$.

Alternatively, with $\un{M} = F(A)$ and $n=-1$ we have $\un{\pi}_k \Sigma^{-\lambda} H F(A) \cong \cO(A)$ for $k=-2$ and trivial for $k \neq -2$ since both $\wh{\ker(R)} = 0$ and $\wh{\coker(R)} = 0$.
\end{proof}

\section{The slice tower for \texorpdfstring{$\Sigma^V H\un{M}$}{suspensions of HM}} \label{s:slicetower}
The starting point is the following two results:
\begin{enumerate}
\item First, $H\un{M}$ is a zero slice for any Mackey functor $\un{M}$. This follows from \cite[Proposition 4.50]{HHR} after adjusting to the modern definition of slices.
 \item Second, smashing with the regular representation sphere $S^{\rho_G}$ simply shifts slices up by $|G|$. This follows from \cite[Lemma 4.23]{HHR}.
\end{enumerate}

We fix $G=C_p$, and note that smashing with $S^{\rho_{C_p}'}$ also shifts slices up by $|G|$.

We start with a few lemmas. The first one is well known, but we include it for completeness:

\begin{lemma} \label{l:hatslice}
For any $m \in \bZ$ and abelian group $A$, $\Sigma^m H \wh{A}$ is a $pm$-slice.
\end{lemma}

\begin{proof}
From the above calculation of the homology of a point we have
\[
 \un{\pi}_k \Sigma^m H \wh{A} \cong \un{\pi}_k \Sigma^{m \rho_{C_p}} H \wh{A} \cong \begin{cases} \wh{A} \quad & \textnormal{if $k=m$} \\ 0 \quad & \textnormal{if $k \neq m$} \end{cases}
\]
Hence $\Sigma^m H \wh{A} \simeq \Sigma^{m \rho_{C_p}} H \wh{A}$.

We know that $H \wh{A}$ is a $0$-slice, and that suspending by a multiple of $\rho_{C_p}$ (or $\rho_{C_p}'$) simply shifts slices up or down, so the result follows.
\end{proof}

\begin{lemma} \label{l:plusslice}
If the restriction map for $\un{M}$ is injective, $\Sigma^{k \lambda} H \un{M}$ is a $2k$-slice for $- \frac{p}{2} \leq k \leq 0$.
\end{lemma}

\begin{proof}
Let $X = \Sigma^{k \lambda} H \un{M}$. The result clearly holds when $k=0$, so we assume $k < 0$. We need to check that $X \leq 2k$ and $X \geq 2k$.

To check that $X \leq 2k$, it suffices to check that the space $Sp^{C_p}(\wh{S}, X)$ of $C_p$-equivariant maps is contractible for each $\wh{S} = (C_p)_+ \sma_H S^{m\rho_H}$ with $\dim(\wh{S}) > 2k$.

This is clear for $H=e$. For $H=C_p$, the conditions on $k$ imply that $m \geq 0$. For $m=0$, we have $\wh{S} = S^0$, so for any $C_p$-spectrum $X$ we get $\pi_q Map(\wh{S}, X) \simeq \un{\pi}_q X (C_p/C_p)$ for $q \geq 0$. The result for $m=0$ then follows from Proposition \ref{p:homologyofpoint} Part (2), since $\pi_0 Sp^{C_p}(S^0, X) \cong \un{H}_0(S^{k \lambda}; \un{M})(C_p/C_p) \cong \ker(R)$ and $\ker(R) = 0$ by assumption.

For $m > 0$ we can use $\rho_{C_p}'$ in place of $\rho_{C_p}$ and we get $Sp^{C_p}(S^{m \rho_{C_p}'}, \Sigma^{k \lambda} H \un{M}) \simeq Sp^{C_p}(S^m, \Sigma^{(k-\frac{m(p-1)}{2})\lambda} H\un{M})$. Since $S^{(k-\frac{m(p-1)}{2})\lambda}$ only has cells in non-positive dimension the result for $m > 0$ follows.

To check that $X \geq 2k$, we can use the alternative characterization from \cite[\S 11.1D]{HHR21}: $X \geq 2k$ if and only if $\pi_m \Phi^H X = 0$ for $m < 2k/|H|$. This is immediate for $H=e$, and for $H=C_p$ the conditions on $k$ imply that $m < 0$. But $\Phi^{C_p}(\Sigma^{k \lambda} H \un{M}) \simeq \Phi^{C_p}(H\un{M})$, so this also follows.
\end{proof}

\begin{corollary} \label{c:plusslice}
Suppose the restriction map for $\un{M}$ is injective. Let $V = m \oplus \bigoplus n_i \lambda_i$, and let $n = \sum n_i$. Define $r = (p-1)m - 2n$. If $r \in [0, p]$ then $\Sigma^V H\un{M}$ is an $(m+2n)$-slice.
\end{corollary}

\begin{lemma} \label{l:0plusslice}
If the transfer map for $\un{M}$ is surjective, $\Sigma^{k \lambda} H \un{M}$ is a $2k$-slice for $0 \leq k \leq \frac{p}{2}$.
\end{lemma}

\begin{proof}
This is similar to the previous lemma, after noting that $\pi_0 \Phi^{C_p}(H\un{M}) = 0$ since $\Phi^{C_p}$ kills transfers and the transfer map is surjective by assumption.
\end{proof}

\begin{corollary} \label{c:0plusslice}
Suppose the transfer map for $\un{M}$ is surjective. Let $V = m \oplus \bigoplus n_i \lambda_i$, and let $n = \sum n_i$. Define $r = (p-1)m - 2n$. If $r \in [-p, 0]$ then $\Sigma^V H\un{M}$ is an $(m+2n)$-slice.
\end{corollary}

The next lemma holds for any finite group $G$, and follow from the properties of a localizing subcategory (see \cite[Proposition 4.45]{HHR}):

\begin{lemma} \label{l:recognizingtower}
Let $X$ be a $G$-spectrum. Suppose we have a tower of fibrations
\[ \xymatrix{
F_0 \ar[d] & F_1 \ar[d] & & F_k \ar[d]^= & \\
X \ar[r] & X_1 \ar[r] & \ldots \ar[r] & X_k \ar[r] & {*}
} \]
where $F_i$ is $n_i$-slice and $n_0 > n_1 > \ldots > n_k$. Then this is the slice tower for $X$.

Similarly, suppose we have a tower of cofibrations
\[ \xymatrix{
 {*} \ar[r] & X^{-k} \ar[r] \ar[d]^= & \ldots \ar[r] & X^{-1} \ar[d] \ar[r] & X \ar[d] \\
 & C^{-k} & & C^{-1} & C^0
} \]
where $C^i$ is $n_i$-slice and $n_0 < n_{-1} < \ldots < n_{-k}$. Then this is the coslice tower for $X$.
\end{lemma}

With this we are ready to describe the slice tower for $\Sigma^V H \un{M}$:

\begin{thm} \label{t:slicetower}
Let $\un{M}$ be a $C_p$-Mackey functor, let $V = m \oplus \bigoplus n_i \lambda_i \in RO(C_p)$, and let $n = \sum n_i$. If $p=2$, allow $n_1=n$ to be a half-integer. Define $r = (p-1)m -2n$. Then the slice tower for $\Sigma^V H\un{M}$ is as follows:
\begin{enumerate}
 \item If $r = 0$, $\Sigma^V H\un{M}$ is an $(m+2n)$-slice.

 \item If $r > 0$, let $k = \lceil \frac{r}{p} \rceil$. Then $\Sigma^V H\un{M} \in [m+2n, pm]$ and the slice tower of $\Sigma^V H\un{M}$ is given by
 \[ \xymatrix @C=0.9pc {
  \Sigma^{m} H \wh{\ker(R)} \ar[d] & \Sigma^{m-1} H \wh{\coker(R)} \ar[d] & \Sigma^{m-2} H\wh{\ker(N_{\un{M}})} \ar[d] & & F_{k-1} \ar[d] & F_k \ar[d]^= & \\
  \Sigma^V H\un{M} \ar[r] & \Sigma^V H \un{\im(R)} \ar[r] & X_2 \ar[r] & \ldots \ar[r] & X_{k-1} \ar[r] & X_k \ar[r] & {*}
 } \]
 Here $X_{2i+2} = \Sigma^{V-2i+i\lambda} HF(\un{M})$ for $i \geq 0$ and $X_{2i+1} = \Sigma^{V-2i+i\lambda} H\un{\im(N_{M})}$ for $i \geq 1$. For $2 \leq i < k$ the fiber $F_i$ of $X_i \to X_{i+1}$ is $\Sigma^{m-i} H \wh{\ker(N_{\un{M}})}$ for $i$ even and $\Sigma^{m-i} H \wh{\coker(N_{\un{M}})}$ for $i$ odd.

 \item If $r < 0$, let $k = -\lceil \frac{-r}{p} \rceil$. Then $\Sigma^V H\un{M} \in [pm, m+2n]$ and the coslice tower of $X$ is given by
 \[ \xymatrix @C=0.9pc {
  {*} \ar[r] & X^k \ar[r] \ar[d]_= & X^{k+1} \ar[r] \ar[d] & \ldots \ar[r] & X^{-2} \ar[r] \ar[d] & \Sigma^V H\un{\im(\tr)} \ar[d] \ar[r] & \Sigma^V H\un{M} \ar[d] \\
  & C^k & C^{k+1} & & \Sigma^{m+2} \wh{\coker(N_{\un{M}})} & \Sigma^{m+1} H \wh{\ker(\tr)} & \Sigma^m H\wh{\coker(\tr)}
 } \]
 Here $X^{2i-2} = \Sigma^{V-2i+i\lambda} H\cO(\un{M})$ for $i \leq 0$ and $X^{2i-1} = \Sigma^{V-2i+i\lambda} H\un{\im(N_{M})}$ for $i \leq -1$. For $k < i \leq -2$ the cofiber $C^i$ of $X^{i-1} \to X^i$ is $\Sigma^{m-i} H \wh{\coker(N_{\un{M}})}$ for $i$ even and $\Sigma^{m-i} H \wh{\ker(N_{\un{M}})}$ for $i$ odd.
\end{enumerate}
\end{thm}

In case (2), the first two fibers ``fix up'' $\Sigma^V H\un{M}$ so that it looks like $F(\un{M})$, and the rest of the slice tower is essentially $2$-periodic. In case (3), the first two cofibers ``fix up'' $\Sigma^V H\un{M}$ so that it looks like $\cO(\un{M})$, and the rest of the coslice tower is essentially $2$-periodic.

\begin{proof}
Part (1) is clear, since in this case $V' = m \oplus n \lambda$ is a multiple of $\rho_{C_p}'$.

For part (2), we start by noting that the leftmost fiber sequence $\Sigma^m H \wh{\ker(R)} \to \Sigma^V H\un{M} \to \Sigma^V H\un{\im(R)}$ can be obtained by applying $\Sigma^V H(-)$ to the first short exact sequence in Example \ref{ex:SESA}. From Lemma \ref{l:hatslice} we know that $\Sigma^m H \wh{\ker(R)}$ is a $pm$-slice.

If $k=1$, it follows from Corollary \ref{c:plusslice} that $\Sigma^V H \un{\im(R)}$ is an $(m+2n)$-slice and from Lemma \ref{l:recognizingtower} that this is the slice tower for $X$ so in that case we are done.

If $k \geq 2$, the next fiber sequence $\Sigma^{m-1} H \wh{\coker(R)} \to \Sigma^V H\un{\im(R)} \to \Sigma^V HF(\un{M})$ can be obtained by applying $\Sigma^V H(-)$ to the second short exact sequence in Example \ref{ex:SESA} and then rotating it. If $k=2$, the conclusion now follows from Corollary \ref{c:0plusslice} and the first half of Lemma \ref{l:recognizingtower} after using the equivalence in Corollary \ref{c:OversusF} to replace $\Sigma^V H F(\un{M})$ with $\Sigma^{V+\lambda-2} H\cO(\un{M})$.

The remaining maps come from the maps $\cO(\un{M}) \to \un{\im(N_M)} \to F(\un{M})$ in Example \ref{ex:fixedandorbit} plus the equivalence in Corollary \ref{c:OversusF}. The general case is similar, using either Corollary \ref{c:plusslice} or Corollary \ref{c:0plusslice} as well as the first half of Lemma \ref{l:recognizingtower}.

Part (3) is similar, now using the second half of Lemma \ref{l:recognizingtower}.
\end{proof}

In case (2) we can obtain the coslice tower from the slice tower by taking fibers of the various maps from $\Sigma^V H\un{M}$, and in case (3) we can obtain the slice tower from the coslice tower by taking cofibers of the various maps to $\Sigma^V H\un{M}$. We can also use Corollary \ref{c:OversusF} to replace the occurrences of $F(\un{M})$ with $\cO(\un{M})$ in (2), or the occurrences of $\cO(\un{M})$ with $F(\un{M})$ in (3). We state the result in this way for maximal symmetry, and because the description of the slice tower is simpler in case (2) and the description of the coslice tower is simpler in case (3).

\begin{corollary} Let $\un{M}$ be a $C_p$-Mackey functor.
 \begin{enumerate}
  \item If $\ker(R) = 0$ then $\Sigma \un{M}$ is a $1$-slice. If $p=2$ then $\Sigma^2 H\un{M}$ is a $2$-slice.
  \item If in addition $\coker \big( R : \un{M}(C_p/C_p) \to \un{M}(C_p/e)^{C_p} \big) = 0$ then $\Sigma^2 H\un{M}$ is a $2$-slice. If $p=2$ then $\Sigma^3 H\un{M}$ is a $3$-slice and $\Sigma^4 H\un{M}$ is a $4$-slice. If $p=3$ then $\Sigma^3 H\un{M}$ is a $3$-slice.
  \item If in addition $\ker \big( N : \un{M}(C_p/e)_{C_p} \to \un{M}(C_p/e)^{C_p} \big) = 0$ then $\Sigma^3 H\un{M}$ is a $3$-slice. If $p=2$ then $\Sigma^5 H\un{M}$ is a $5$-slice and $\Sigma^6 H\un{M}$ is a $6$-slice. If $p=3$ then $\Sigma^4 H\un{M}$ is a $4$-slice.
  \item If in addition $\coker \big( N : \un{M}(C_p/e)_{C_p} \to \un{M}(C_p/e)^{C_p} \big) = 0$ then $\Sigma^n H\un{M}$ is an $n$-slice for all $n \geq 0$.
 \end{enumerate}

\end{corollary}

For example, at $p=2$ we see that $\Sigma^6 H \un{\bZ}$ is a $6$-slice.

\begin{corollary} Let $\un{M}$ be a $C_p$-Mackey functor.
 \begin{enumerate}
  \item If $\coker(\tr) = 0$ then $\Sigma^{-1} \un{M}$ is a $(-1)$-slice. If $p=2$ then $\Sigma^{-2} H\un{M}$ is a $(-2)$-slice.
  \item If in addition $\ker \big( \tr : \un{M}(C_p/e)_{C_p} \to \un{M}(C_p/C_p) \big) = 0$ then $\Sigma^{-2} H\un{M}$ is a $(-2)$-slice. If $p=2$, $\Sigma^{-3} H\un{M}$ is a $(-3)$-slice and $\Sigma^{-4} H\un{M}$ is a $(-4)$-slice. If $p=3$ then $\Sigma^{-3} H\un{M}$ is a $(-3)$-slice.
  \item If in addition $\coker \big( N : \un{M}(C_p/e)_{C_p} \to \un{M}(C_p/e)^{C_p} \big) = 0$ then $\Sigma^{-3} H\un{M}$ is a $(-3)$-slice. If $p=2$ then $\Sigma^{-5} H\un{M}$ is a $(-5)$-slice and $\Sigma^{-6} H\un{M}$ is a $(-6)$-slice. If $p=3$ then $\Sigma^{-4} H\un{M}$ is a $(-4)$-slice.
  \item If in addition $\ker \big( N : \un{M}(C_p/e)_{C_p} \to \un{M}(C_p/e)^{C_p} \big) = 0$ then $\Sigma^n H\un{M}$ is an $n$-slice for all $n \leq 0$.
 \end{enumerate}

\end{corollary}

For example, at $p=2$ we see that $\Sigma^{-6} H \un{\bQ/\bZ}^*$ is a $(-6)$-slice. This also follows from Brown-Comenetz duality plus the above observation that $\Sigma^6 H\un{\bZ}$ is a $6$-slice.

\begin{example}
Consider $\Sigma^V H\un{\bZ}$. Then the slice tower simplifies, because $\un{\im(N_{\bZ})} = \un{\bZ}^*$ and $\wh{\ker(N_{\bZ})} = 0$ while $\wh{\coker(N_{\bZ})} = \wh{\bZ/p}$.

Define $k$ as in Theorem \ref{t:slicetower}. If $k \in [0,3]$ then $\Sigma^V H\un{\bZ}$ is $(m+2n)$-slice. For $k \geq 4$, let $k' = \lceil \frac{k-3}{2} \rceil$. Then $\Sigma^V H\un{\bZ} \in [m+2n, p(m-3)]$ and the slice tower for $\Sigma^V H \un{\bZ}$ looks as follows:
\[ \xymatrix @C=1.2pc {
  \Sigma^{m-3} H \wh{\bZ/p} \ar[d] & \Sigma^{m-5} H \wh{\bZ/p} \ar[d] & & F_{k'-1} \ar[d] & F_{k'} \ar[d]^= & \\
  \Sigma^V H\un{\bZ} \ar[r] & \Sigma^{V-2+\lambda} H \un{\bZ} \ar[r] & \ldots \ar[r] & X_{k'-1} \ar[r] & X_{k'} \ar[r] & {*}
 } \]
Here $X_i = \Sigma^{V-2i+i\lambda} H\un{\bZ}$ and $F_i = \Sigma^{m-3-2i} \wh{\bZ/p}$ for $0 \leq i < k'$.

For $k < 0$ the coslice tower simplifies in a similar way. This time we let $k' = \lfloor \frac{k}{2} \rfloor$. Then $\Sigma^V H\un{\bZ} \in [pm, m+2n]$ and the coslice tower for $\Sigma^V H\un{\bZ}$ looks as follows:
\[ \xymatrix @C=1.2pc {
  {*} \ar[r] & X^{k'} \ar[r] \ar[d]_= & X^{k'+1} \ar[r] \ar[d] & \ldots \ar[r] & \Sigma^{V+2-\lambda} H\un{\bZ} \ar[d] \ar[r] & \Sigma^V H\un{\bZ} \ar[d] \\
  & C^{k'} & C^{k'+1} & & \Sigma^{m+2} H \wh{\bZ/p} & \Sigma^m H\wh{\bZ/p}
 } \]
This time $X^i = \Sigma^{V-2i+i\lambda} H\un{\bZ}$ and $C^i = \Sigma^{m-2i} H\wh{\bZ/p}$ for $k' < i \leq 0$.
\end{example}

\begin{example}
Consider $\Sigma^V H\un{\bQ/\bZ}^*$. This is Brown-Comenetz dual to $\Sigma^{-V} H\un{\bZ}$, so the slice tower for $\Sigma^V H\un{\bQ/\bZ}$ is Brown-Comentz dual to the coslice tower for $\Sigma^{-V} H\un{\bZ}$ and the coslice tower for $\Sigma^V H\un{\bQ/\bZ}$ is Brown-Comenetz dual to the slice tower for $\Sigma^{-V} H\un{\bZ}$.

Define $k$ as in Theorem \ref{t:slicetower}. If $k \in [-3, 0]$ then $\Sigma^V H\un{\bQ/\bZ}^*$ is $(m+2n)$-slice. $k \leq -4$ we let $k' = \lfloor \frac{k+3}{2} \rfloor$. Then $\Sigma^V H\un{\bQ/\bZ}^* \in [p(m+3), m+2n]$ and the coslice tower for $\Sigma^V H\un{\bZ/\bZ}^*$ looks as follows:
\[ \xymatrix @C=1.2pc {
  {*} \ar[r] & X^{k'} \ar[r] \ar[d]_= & X^{k'+1} \ar[r] \ar[d] & \ldots \ar[r] & \Sigma^{V+2-\lambda} H\un{\bQ/\bZ}^* \ar[d] \ar[r] & \Sigma^V H\un{\bQ/\bZ}^* \ar[d] \\
  & C^{k'} & C^{k'+1} & & \Sigma^{m+5} H \wh{\bZ/p} & \Sigma^{m+3} H\wh{\bZ/p}
 } \]
Here $X^i = \Sigma^{V-2i+i\lambda} H \un{\bQ/\bZ}^*$ and $C^i = \Sigma^{m+3-2i} H\wh{\bZ/p}$ for $k' < i \leq 0$.

For $k > 0$ the slice tower simplifies in a similar way. This time we let $k' = \lceil \frac{k}{2} \rceil$. Then $\Sigma^V H\un{\bZ} \in [m+2n, pm]$ and the slice tower for $\Sigma^V H\un{\bQ/\bZ}^*$ looks as follows:
\[ \xymatrix @C=1.2pc {
  \Sigma^m H \wh{\bZ/p} \ar[d] & \Sigma^{m-2} H \wh{\bZ/p} \ar[d] & & F_{k'-1} \ar[d] & F_{k'} \ar[d]^= & \\
  \Sigma^V H\un{\bQ/\bZ}^* \ar[r] & \Sigma^{V-2+\lambda} H \un{\bQ/\bZ}^* \ar[r] & \ldots \ar[r] & X_{k'-1} \ar[r] & X_{k'} \ar[r] & {*}
 } \]
Here $X_i = \Sigma^{V-2i+i\lambda} H\un{\bQ/\bZ}^*$ and $F_i = \Sigma^{m-2i} \wh{\bZ/p}$ for $0 \leq i < k'$.
\end{example}

\section{The slice spectral sequence for \texorpdfstring{$\Sigma^V H\un{M}$}{suspensions of HM}} \label{s:examples}
This section is longer than is logically necessary, but for the reader's convenience we include a large number of pictures to make it easy to see how the slice spectral sequence changes as we vary $n$ and $k$.

As above we let $V = m \oplus \bigoplus n_i \lambda(i)$ and $n = \sum n_i$. When $n$ is an integer (as opposed to a half-integer, which can only happen at $p=2$) we will use the following symbols for the various Mackey functors appearing in the slice spectral sequence:
\[
\arraycolsep=6pt\def\arraystretch{2.0}
\begin{array}{cccc}
 R = \un{\im(R)} &
 T =\un{\im(\tr)} &
 F = F(\un{M}) &
 \cO = \cO(\un{M}) \\
 \circ = \wh{\ker(N_{\un{M}})} &
 \bullet = \wh{\coker(N_{\un{M}})} &
 \overline{\circ} = \wh{\ker(R)} &
 \overline{\bullet} = \wh{\coker(R)} \\
 \underline{\circ} = \wh{\ker(\tr)} &
 \underline{\bullet} = \wh{\coker(\tr)} &
 \blacktriangle = \wh{\im(R)/\im(N_{\un{M}})} &
 \triangledown = \wh{\ker(N_{\un{M}})/\ker(\tr)}
\end{array}
\]

As in \cite{HHR} we draw the slice spectral sequence with $t-s$ (the total degree) on the horizontal axis and $s$ on the vertical axis. The general shape of the spectral sequence for $\Sigma^V H \un{M}$ depends only on $n$ and $k$ (see Theorem \ref{t:slicetower}), while the exact position of each Mackey functor and the exact length of each differential and extension depend on additional variables. For this reason, when we draw the slice spectral sequence we omit the labels on the axes. To save space we also omit the axes themselves.

For fixed $k > 0$ and $n \in \bZ$, $m$ is determined by
\[
 \frac{p(k-1)+2n}{p-1} < m \leq \frac{pk+2n}{p-1}.
\]
The half-open interval $(\frac{p(k-1)+2n}{p-1}, \frac{pk+2n}{p-1} ]$ has length $\frac{p}{p-1}$. Hence we find that for $p=2$ there are always $2$ possible values of $m$. For $p$ odd there are always at least $1$ possible value of $m$, and sometimes $2$. The same holds for $k < 0$. Hence each of the spectral sequences we are about to describe is indeed the slice spectral sequence for $\Sigma^V H\un{M}$ for some $V$.

The next 5 examples cover $n=0$, $n=1$, $n=-1$, $n \geq 2$ and $n \leq -2$ respectively, so together they cover all possible values of $n$ as long as $p \neq 2$.

\begin{example}
Consider $\Sigma^m H\un{M}$ for various $m$. When $k=0$ the slice spectral sequence has a single non-trivial Mackey functor: $\un{\cA}^a \Box \un{M}$ for some $a$ in filtration $m$ and degree $m$.

When $k=1$, the slice spectral sequence has two non-trivial Mackey functors and a hidden extension. The extension encodes the first short exact sequence in Example \ref{ex:SESA}, or some twisted version of it $0 \to \overline{\circ} \to \cA^a \Box \un{M} \to R \to 0$:

\begin{center}
\begin{sseqdata}[ no axes, struct lines = blue, name = exn0k1, xscale = 0.6, yscale = 0.3, Adams grading, classes = {draw = none} ]
\class["R"](13,0)
\class["{\overline{\circ}}"](13,4)
\structline(13,0)(13,4)
\end{sseqdata}
\printpage[ name = exn0k1 ]
\end{center}

When $k=2$, the slice spectral sequence has a differential and a hidden extension. The differential is surjective and the kernel is $\un{\im(R)}$. This represents the short exact sequence $0 \to R \to F \to \overline{\bullet} \to 0$ in Example \ref{ex:SESA}. The hidden extension then puts $\wh{\ker(R)}$ back in, the same way as for $k=1$:

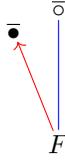
\begin{center}
\begin{sseqdata}[ no axes, struct lines = blue, name = exn0k2, xscale = 0.6, yscale = 0.3, Adams grading, classes = {draw = none} ]
\class["F"](13,0)
\class["{\overline{\circ}}"](13,6)
\class["{\overline{\bullet}}"](12,5)
\structline(13,0)(13,6)
\d[red]5(13,0)
% \d[red]8(11,2)
% \d[red]12(13,0)
\end{sseqdata}
\printpage[ name = exn0k2 ]
\end{center}

We could splice the two short exact sequences in Example 4.4 and say that the spectral sequence represents the $4$-term exact sequence
\[
 0 \to \overline{\circ} \to \un{\cA}^a \Box \un{M} \to F \to \overline{\bullet} \to 0
\]

If $k=3$, the slice spectral sequence picks up two more classes and another differential:

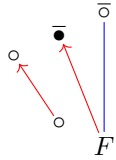
\begin{center}
\begin{sseqdata}[ no axes, struct lines = blue, name = exn0k3, xscale = 0.6, yscale = 0.3, Adams grading, classes = {draw = none} ]
\class["F"](13,0)
\class["{\circ}"](12,1)
\class["{\overline{\circ}}"](13,6)
\class["{\overline{\bullet}}"](12,5)
\class["{\circ}"](11,4)
\structline(13,0)(13,6)
\d[red]3(12,1)
\d[red]5(13,0)
\end{sseqdata}
\printpage[ name = exn0k3 ]
\end{center}

For higher values of $k$, the spectral sequence simply picks up additional classes and differentials. For $k=4$ it looks as follows:

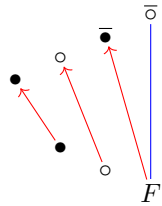
\begin{center}
\begin{sseqdata}[ no axes, struct lines = blue, name = exn0k4, xscale = 0.6, yscale = 0.3, Adams grading, classes = {draw = none} ]
\class["F"](13,0)
\class["{\circ}"](12,1)
\class["{\bullet}"](11,2)
\class["{\overline{\circ}}"](13,8)
\class["{\overline{\bullet}}"](12,7)
\class["{\circ}"](11,6)
\class["{\bullet}"](10,5)
\structline(13,0)(13,8)
\d[red]3(11,2)
\d[red]5(12,1)
\d[red]7(13,0)
\end{sseqdata}
\printpage[ name = exn0k4 ]
\end{center}

Each time we increase $k$ by one, the existing differentials and the extension get longer and we put in an additional differential $\circ \to \circ$ or $\bullet \to \bullet$ on the left hand side. This happens for any $V$, as soon as $k$ is sufficiently large.

When $k$ is negative the slice spectral sequence looks slightly different. When $k=-1$ we have a single hidden extension, encoding a twisted version of the first short exact sequence in Example \ref{ex:SESB}:
\begin{center}
\begin{sseqdata}[ no axes, struct lines = blue, name = exn0k-1, xscale = 0.6, yscale = 0.3, Adams grading, classes = {draw = none} ]
\class["\underline{\bullet}"](13,0)
\class["T"](13,6)
\structline(13,0)(13,6)
\end{sseqdata}
\printpage[ name = exn0k-1 ]
\end{center}

When $k=-2$ we have one differential and one extension, encoding the two short exact sequences in Example \ref{ex:SESB}:
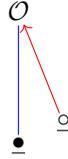
\begin{center}
\begin{sseqdata}[ no axes, struct lines = blue, name = exn0k-2, xscale = 0.6, yscale = 0.3, Adams grading, classes = {draw = none} ]
\class["\underline{\bullet}"](13,0)
\class["\underline{\circ}"](14,1)
\class["\cO"](13,6)
\structline(13,0)(13,6)
\d[red]5(14,1)
\end{sseqdata}
\printpage[ name = exn0k-2 ]
\end{center}

Again we can splice the two short exact sequences in Example \ref{ex:SESB} and say that the spectral sequence represents the $4$-term exact sequence
\[
 0 \to \underline{\circ} \to \cO \to \un{\cA}^a \Box \un{M} \to \underline{\bullet} \to 0
\]

For lower values of $k$, the spectral sequence simply picks up additional classes and differentials. For $k=-3$ it looks as follows:
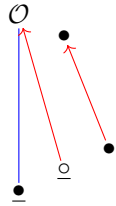
\begin{center}
\begin{sseqdata}[ no axes, struct lines = blue, name = exn0k-3, xscale = 0.6, yscale = 0.3, Adams grading, classes = {draw = none} ]
\class["\underline{\bullet}"](13,0)
\class["\underline{\circ}"](14,1)
\class["{\bullet}"](15,2)
\class["\cO"](13,8)
\class["{\bullet}"](14,7)
\structline(13,0)(13,8)
\d[red]5(15,2)
\d[red]7(14,1)
\end{sseqdata}
\printpage[ name = exn0k-3 ]
\end{center}

Each time we decrease $k$ by one, the existing differentials and the extension get longer and we put in an additional differential $\circ \to \circ$ or $\bullet \to \bullet$ on the right hand side. This happens for any $V$, as soon as $k$ is sufficiently small.

\end{example}

\begin{example}
Consider $\Sigma^{m+\lambda} H\un{M}$ for various $m$. When $k=0$, everything is concentrated in a single slice:
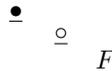
\begin{center}
\begin{sseqdata}[ no axes, struct lines = blue, name = exn1k0, xscale = 0.6, yscale = 0.3, Adams grading, classes = {draw = none} ]
\class["F"](13,0)
\class["\underline{\circ}"](12,1)
\class["\underline{\bullet}"](11,2)
\end{sseqdata}
\printpage[ name = exn1k0 ]
\end{center}

When $k=1$ we pick up a differential and an extension. These can, together, be encoded by the exact sequence
\[
 0 \to \underline{\circ} \to \circ \to \overline{\circ} \to \underline{\bullet} \to \blacktriangle \to 0
\]

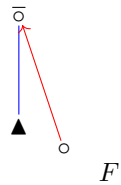
\begin{center}
\begin{sseqdata}[ no axes, struct lines = blue, name = exn1k1, xscale = 0.6, yscale = 0.3, Adams grading, classes = {draw = none} ]
\class["F"](13,0)
\class["\circ"](12,1)
\class["{\blacktriangle}"](11,2)
\class["\overline{\circ}"](11,7)
\structline(11,2)(11,7)
\d[red]6(12,1)
\end{sseqdata}
\printpage[ name = exn1k1 ]
\end{center}

When $k=2$ we pick up another differential, so the whole spectral sequence encodes the exact sequence
\[
0 \to \underline{\circ} \to \circ \to \overline{\circ} \to \underline{\bullet} \to \bullet \to \overline{\bullet} \to 0
\]
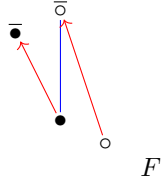
\begin{center}
\begin{sseqdata}[ no axes, struct lines = blue, name = exn1k2, xscale = 0.6, yscale = 0.3, Adams grading, classes = {draw = none} ]
\class["F"](13,0)
\class["\circ"](12,1)
\class["\bullet"](11,2)
\class["\overline{\circ}"](11,7)
\class["\overline{\bullet}"](10,6)
\structline(11,2)(11,7)
\d[red]4(11,2)
\d[red]6(12,1)
\end{sseqdata}
\printpage[ name = exn1k2 ]
\end{center}

For higher values of $k$ we simply pick up additional classes and differentials as in the previous example: For $k=3$ it looks as follows:
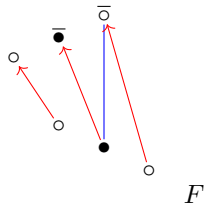
\begin{center}
\begin{sseqdata}[ no axes, struct lines = blue, name = exn1k3, xscale = 0.6, yscale = 0.3, Adams grading, classes = {draw = none} ]
\class["F"](13,0)
\class["\circ"](12,1)
\class["\bullet"](11,2)
\class["\circ"](10,3)
\class["\overline{\circ}"](11,8)
\class["\overline{\bullet}"](10,7)
\class["\circ"](9,6)
\structline(11,2)(11,8)
\d[red]3(10,3)
\d[red]5(11,2)
\d[red]7(12,1)
\end{sseqdata}
\printpage[ name = exn1k3 ]
\end{center}

When $k$ is negative the spectral sequence again looks a bit different. When $k=-1$ the spectral sequence is concentrated in two slice filtrations, but there are no differentials or extensions:
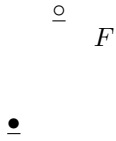
\begin{center}
\begin{sseqdata}[ no axes, struct lines = blue, name = exn1k-1, xscale = 0.6, yscale = 0.3, Adams grading, classes = {draw = none} ]
\class["F"](13,6)
\class["\underline{\circ}"](12,7)
\class["\underline{\bullet}"](11,2)
\end{sseqdata}
\printpage[ name = exn1k-1 ]
\end{center}

When $k=-2$ there are again no differentials or extensions:
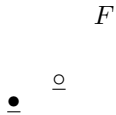
\begin{center}
\begin{sseqdata}[ no axes, struct lines = blue, name = exn1k-2, xscale = 0.6, yscale = 0.3, Adams grading, classes = {draw = none} ]
\class["F"](13,6)
\class["\underline{\circ}"](12,3)
\class["\underline{\bullet}"](11,2)
\end{sseqdata}
\printpage[ name = exn1k-2 ]
\end{center}

When $k=-3$ we pick up an extension, encoding the short exact sequence $0 \to T \to F \to \bullet \to 0$:
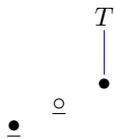
\begin{center}
\begin{sseqdata}[ no axes, struct lines = blue, name = exn1k-3, xscale = 0.6, yscale = 0.3, Adams grading, classes = {draw = none} ]
\class["T"](13,7)
\class["\bullet"](13,4)
\class["\underline{\circ}"](12,3)
\class["\underline{\bullet}"](11,2)
\structline(13,4)(13,7)
\end{sseqdata}
\printpage[ name = exn1k-3 ]
\end{center}

When $k=-4$ we pick up a differential as well:
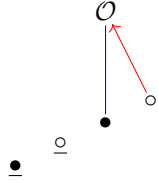
\begin{center}
\begin{sseqdata}[ no axes, struct lines = blue, name = exn1k-4, xscale = 0.6, yscale = 0.3, Adams grading, classes = {draw = none} ]
\class["\cO"](13,9)
\class["\circ"](14,5)
\class["\bullet"](13,4)
\class["\underline{\circ}"](12,3)
\class["\underline{\bullet}"](11,2)
\structline(13,4)(13,9)
\d[red]4(14,5)
\end{sseqdata}
\printpage[ name = exn1k-4 ]
\end{center}

For lower values of $k$ we simply pick up additional classes and differentials. For $k=-5$ it looks as follows:
\begin{center}
\begin{sseqdata}[ no axes, struct lines = blue, name = exn1k-5, xscale = 0.6, yscale = 0.3, Adams grading, classes = {draw = none} ]
\class["\cO"](13,9)
\class["\bullet"](14,8)

\class["\bullet"](15,4)
\class["\circ"](14,3)
\class["\bullet"](13,2)
\class["\underline{\circ}"](12,1)
\class["\underline{\bullet}"](11,0)
\structline(13,2)(13,9)
\d[red]4(15,4)
\d[red]6(14,3)
\end{sseqdata}
\printpage[ name = exn1k-5 ]
\end{center}

\end{example}

\begin{example}
Consider $\Sigma^{m-\lambda} H\un{M}$ for various $m$. When $k=0$, everything is concentrated in a single slice:
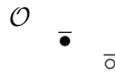
\begin{center}
\begin{sseqdata}[ no axes, struct lines = blue, name = exn-1k0, xscale = 0.6, yscale = 0.3, Adams grading, classes = {draw = none} ]
\class["\cO"](-13,0)
\class["\overline{\bullet}"](-12,-1)
\class["\overline{\circ}"](-11,-2)
\end{sseqdata}
\printpage[ name = exn-1k0 ]
\end{center}

When $k=-1$ we pick up a differential and an extension:
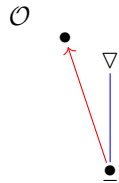
\begin{center}
\begin{sseqdata}[ no axes, struct lines = blue, name = exn-1k-1, xscale = 0.6, yscale = 0.3, Adams grading, classes = {draw = none} ]
\class["\cO"](-13,0)
\class["\bullet"](-12,-1)
\class["{\triangledown}"](-11,-2)
\class["\underline{\bullet}"](-11,-7)
\structline(-11,-2)(-11,-7)
\d[red]6(-11,-7)
\end{sseqdata}
\printpage[ name = exn-1k-1 ]
\end{center}

When $k=-2$ we pick up another differential:
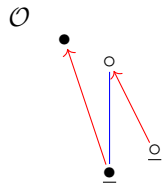
\begin{center}
\begin{sseqdata}[ no axes, struct lines = blue, name = exn-1k-2, xscale = 0.6, yscale = 0.3, Adams grading, classes = {draw = none} ]
\class["\cO"](-13,0)
\class["\bullet"](-12,-1)
\class["\circ"](-11,-2)
\class["\underline{\bullet}"](-11,-7)
\class["\underline{\circ}"](-10,-6)
\structline(-11,-2)(-11,-7)
\d[red]4(-10,-6)
\d[red]6(-11,-7)
\end{sseqdata}
\printpage[ name = exn-1k-2 ]
\end{center}

For lower values of $k$ we simply pick up additional classes and differentials: For $k=-3$ it looks as follows:
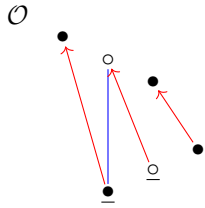
\begin{center}
\begin{sseqdata}[ no axes, struct lines = blue, name = exn-1k-3, xscale = 0.6, yscale = 0.3, Adams grading, classes = {draw = none} ]
\class["\cO"](-13,0)
\class["\bullet"](-12,-1)
\class["\circ"](-11,-2)
\class["\bullet"](-10,-3)
\class["\underline{\bullet}"](-11,-8)
\class["\underline{\circ}"](-10,-7)
\class["\bullet"](-9,-6)
\structline(-11,-2)(-11,-8)
\d[red]3(-9,-6)
\d[red]5(-10,-7)
\d[red]7(-11,-8)
\end{sseqdata}
\printpage[ name = exn-1k-3 ]
\end{center}

When $k$ is positive the spectral sequence again looks a bit different. When $k=1$ the spectral sequence is concentrated in two slice filtrations, but there are no differentials or extensions:
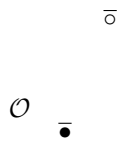
\begin{center}
\begin{sseqdata}[ no axes, struct lines = blue, name = exn-1k1, xscale = 0.6, yscale = 0.3, Adams grading, classes = {draw = none} ]
\class["\cO"](-13,-6)
\class["\overline{\bullet}"](-12,-7)
\class["\overline{\circ}"](-11,-2)
\end{sseqdata}
\printpage[ name = exn-1k1 ]
\end{center}

When $k=2$ there are again no differentials or extensions:
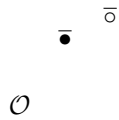
\begin{center}
\begin{sseqdata}[ no axes, struct lines = blue, name = exn-1k2, xscale = 0.6, yscale = 0.3, Adams grading, classes = {draw = none} ]
\class["\cO"](-13,-6)
\class["\overline{\bullet}"](-12,-3)
\class["\overline{\circ}"](-11,-2)
\end{sseqdata}
\printpage[ name = exn-1k2 ]
\end{center}

When $k=3$ we pick up an extension, encoding the short exact sequence $0 \to \circ \to \cO \to R \to 0$:
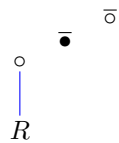
\begin{center}
\begin{sseqdata}[ no axes, struct lines = blue, name = exn-1k3, xscale = 0.6, yscale = 0.3, Adams grading, classes = {draw = none} ]
\class["R"](-13,-7)
\class["\circ"](-13,-4)
\class["\overline{\bullet}"](-12,-3)
\class["\overline{\circ}"](-11,-2)
\structline(-13,-4)(-13,-7)
\end{sseqdata}
\printpage[ name = exn-1k3 ]
\end{center}

When $k=4$ we pick up a differential as well:
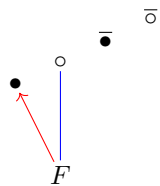
\begin{center}
\begin{sseqdata}[ no axes, struct lines = blue, name = exn-1k4, xscale = 0.6, yscale = 0.3, Adams grading, classes = {draw = none} ]
\class["F"](-13,-9)
\class["\bullet"](-14,-5)
\class["\circ"](-13,-4)
\class["\overline{\bullet}"](-12,-3)
\class["\overline{\circ}"](-11,-2)
\structline(-13,-4)(-13,-9)
\d[red]4(-13,-9)
\end{sseqdata}
\printpage[ name = exn-1k4 ]
\end{center}

For higher values of $k$ we simply pick up additional classes and differentials. For $k=5$ it looks as follows:
\begin{center}
\begin{sseqdata}[ no axes, struct lines = blue, name = exn-1k5, xscale = 0.6, yscale = 0.3, Adams grading, classes = {draw = none} ]
\class["F"](-13,-9)
\class["\circ"](-14,-8)
\class["\circ"](-15,-4)
\class["\bullet"](-14,-3)
\class["\circ"](-13,-2)
\class["\overline{\bullet}"](-12,-1)
\class["\overline{\circ}"](-11,0)
\structline(-13,-2)(-13,-9)
\d[red]4(-14,-8)
\d[red]6(-13,-9)
\end{sseqdata}
\printpage[ name = exn-1k5 ]
\end{center}

\end{example}

\begin{example}
Consider $\Sigma^{m+n\lambda} H\un{M}$ for $n \geq 2$. For $k \geq 0$ the spectral sequence is identical to the one for $\Sigma^{m+\lambda} H\un{M}$, except with $n-1$ additional classes $\circ$ and $n-1$ additional classes $\bullet$. For example, when $n=2$ and $k=2$ the slice spectral sequence looks as follows:
\begin{center}
\begin{sseqdata}[ no axes, struct lines = blue, name = exn2k2, xscale = 0.6, yscale = 0.3, Adams grading, classes = {draw = none} ]
\class["F"](15,-2)
\class["\circ"](14,-1)
\class["\bullet"](13,0)
\class["\circ"](12,1)
\class["\bullet"](11,2)
\class["\overline{\circ}"](11,7)
\class["\overline{\bullet}"](10,6)
\structline(11,2)(11,7)
\d[red]4(11,2)
\d[red]6(12,1)
\end{sseqdata}
\printpage[ name = exn2k2 ]
\end{center}

For $k$ between $-1$ and $-2n$, the slice spectral sequence has no extensions or differentials. For example, when $n=2$ and $k=-3$ the slice spectral sequence looks as follows:
\begin{center}
\begin{sseqdata}[ no axes, struct lines = blue, name = exn2k-3, xscale = 0.6, yscale = 0.3, Adams grading, classes = {draw = none} ]
\class["F"](13,6)
\class["\circ"](12,7)
\class["\bullet"](11,2)
\class["\underline{\circ}"](10,1)
\class["\underline{\bullet}"](9,0)
\end{sseqdata}
\printpage[ name = exn2k-3 ]
\end{center}

For $k \leq -2n-1$ the slice spectral sequence looks like the one for $n=1$, except with $n-1$ additional classes $\circ$ and $n-1$ additional classes $\bullet$. For example, when $n=2$ and $k=-6$ the slice spectral sequence looks as follows:
\begin{center}
\begin{sseqdata}[ no axes, struct lines = blue, name = exn2k-6, xscale = 0.6, yscale = 0.3, Adams grading, classes = {draw = none} ]
\class["\cO"](13,9)
\class["\circ"](14,5)
\class["\bullet"](13,4)
\class["\circ"](12,3)
\class["\bullet"](11,2)
\class["\underline{\circ}"](10,1)
\class["\underline{\bullet}"](9,0)
\structline(13,4)(13,9)
\d[red]4(14,5)
\end{sseqdata}
\printpage[ name = exn2k-6 ]
\end{center}

\end{example}

\begin{example}
Consider $\Sigma^{m+n\lambda} H\un{M}$ for $n \leq -2$. For $k \leq 0$ the spectral sequence is identical to the one for $\Sigma^{m-\lambda} H\un{M}$, except with $-n-1$ additional classes $\circ$ and $-n-1$ additional classes $\bullet$. For example, when $n=-2$ and $k=-2$ the slice spectral sequence looks as follows:
\begin{center}
\begin{sseqdata}[ no axes, struct lines = blue, name = exn-2k-2, xscale = 0.6, yscale = 0.3, Adams grading, classes = {draw = none} ]
\class["\cO"](-15,2)
\class["\bullet"](-14,1)
\class["\circ"](-13,0)
\class["\bullet"](-12,-1)
\class["\circ"](-11,-2)
\class["\underline{\bullet}"](-11,-7)
\class["\underline{\circ}"](-10,-6)
\structline(-11,-2)(-11,-7)
\d[red]4(-10,-6)
\d[red]6(-11,-7)
\end{sseqdata}
\printpage[ name = exn-2k-2 ]
\end{center}

For $k$ between $1$ and $-2n$, the slice spectral sequence has no extensions or differentials. For example, when $n=-2$ and $k=3$ the slice spectral sequence looks as follows:
\begin{center}
\begin{sseqdata}[ no axes, struct lines = blue, name = exn-2k3, xscale = 0.6, yscale = 0.3, Adams grading, classes = {draw = none} ]
\class["\cO"](-13,-6)
\class["\bullet"](-12,-7)
\class["\circ"](-11,-2)
\class["\overline{\bullet}"](-10,-1)
\class["\overline{\circ}"](-9,0)
\end{sseqdata}
\printpage[ name = exn-2k3 ]
\end{center}

For $k \geq -2n+1$ the slice spectral sequence looks like the one for $n=-1$, except with $-n-1$ additional classes $\circ$ and $-n-1$ additional classes $\bullet$. For example, when $n=-2$ and $k=6$ the slice spectral sequence looks as follows:
\begin{center}
\begin{sseqdata}[ no axes, struct lines = blue, name = exn-2k6, xscale = 0.6, yscale = 0.3, Adams grading, classes = {draw = none} ]
\class["F"](-13,-9)
\class["\bullet"](-14,-5)
\class["\circ"](-13,-4)
\class["\bullet"](-12,-3)
\class["\circ"](-11,-2)
\class["\overline{\bullet}"](-10,-1)
\class["\overline{\circ}"](-9,0)
\structline(-13,-4)(-13,-9)
\d[red]4(-13,-9)
\end{sseqdata}
\printpage[ name = exn-2k6 ]
\end{center}

\end{example}

To complete the description of the slice spectral sequence we also need to consider the case when $p=2$ and $n$ is a half-integer.

To do this we need a few more Mackey functors, as defined in Example \ref{ex:p=2MFs}:
\begin{eqnarray*}
 M^- & = & \un{\ker(\tr)}_- \qquad \textnormal{for $\un{M}$} \\
 F^- & = & \un{\ker(\tr)}_- \qquad \textnormal{for $F(\un{M})$} \\
 \cO^- & = & \un{\ker(\tr)}_- \qquad \textnormal{for $\cO(\un{M})$} \\
 M_- & = & \un{\coker(R)}_- \qquad \textnormal{for $\un{M}$} \\
 F_- & = & \un{\coker(R)}_- \qquad \textnormal{for $F(\un{M})$} \\
 \cO_- & = & \un{\coker(R)}_- \qquad \textnormal{for $\cO(\un{M})$} \\
\end{eqnarray*}

\begin{example}
When $p=2$, consider $\Sigma^{m+\frac{1}{2} \lambda} H \un{M}$. When $k=0$, everything is concentrated in a single slice:
\begin{center}
\begin{sseqdata}[ no axes, struct lines = blue, name = exn0.5k0, xscale = 0.6, yscale = 0.3, Adams grading, classes = {draw = none} ]
\class["M^-"](13,6)
\class["\underline{\bullet}"](12,7)
\end{sseqdata}
\printpage[ name = exn0.5k0 ]
\end{center}

When $k=1$, we pick up a differential and an extension:
\begin{center}
\begin{sseqdata}[ no axes, struct lines = blue, name = exn0.5k1, xscale = 0.6, yscale = 0.3, Adams grading, classes = {draw = none} ]
\class["F^-"](13,6)
\class["\blacktriangle"](12,7)
\class["\overline{\circ}"](12,11)
\structline(12,7)(12,11)
\d[red]5(13,6)
\end{sseqdata}
\printpage[ name = exn0.5k1 ]
\end{center}

When $k=2$, we pick up one more differential:
\begin{center}
\begin{sseqdata}[ no axes, struct lines = blue, name = exn0.5k2, xscale = 0.6, yscale = 0.3, Adams grading, classes = {draw = none} ]
\class["F^-"](13,4)
\class["\bullet"](12,5)
\class["\overline{\circ}"](12,11)
\class["\overline{\bullet}"](11,10)
\structline(12,5)(12,11)
\d[red]5(12,5)
\d[red]7(13,4)
\end{sseqdata}
\printpage[ name = exn0.5k2 ]
\end{center}

For larger values of $k$, we simply pick up additional classess and differentials as in the previous examples. For $k=3$ it looks as follows:
\begin{center}
\begin{sseqdata}[ no axes, struct lines = blue, name = exn0.5k3, xscale = 0.6, yscale = 0.3, Adams grading, classes = {draw = none} ]
\class["F^-"](13,2)
\class["\bullet"](12,3)
\class["\overline{\circ}"](12,11)
\class["\overline{\bullet}"](11,10)
\class["\circ"](11,4)
\class["\circ"](10,9)
\structline(12,3)(12,11)
\d[red]5(11,4)
\d[red]7(12,3)
\d[red]9(13,2)
\end{sseqdata}
\printpage[ name = exn0.5k3 ]
\end{center}

For $k=-1$ we have no differentials or extensions:
\begin{center}
\begin{sseqdata}[ no axes, struct lines = blue, name = exn0.5k-1, xscale = 0.6, yscale = 0.3, Adams grading, classes = {draw = none} ]
\class["M^-"](13,6)
\class["\underline{\bullet}"](12,2)
\end{sseqdata}
\printpage[ name = exn0.5k-1 ]
\end{center}

For $k=-2$ we pick up a single extension encoding the short exact sequence $0 \to \underline{\circ} \to M^- \to \cO^- \to 0$:
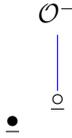
\begin{center}
\begin{sseqdata}[ no axes, struct lines = blue, name = exn0.5k-2, xscale = 0.6, yscale = 0.3, Adams grading, classes = {draw = none} ]
\class["\cO^-"](13,6)
\class["\underline{\bullet}"](12,1)
\class["\underline{\circ}"](13,2)
\structline(13,2)(13,6)
\end{sseqdata}
\printpage[ name = exn0.5k-2 ]
\end{center}

For $k=-3$ we pick up a differential as well, and together with the differential these encode the $4$-term exact sequence $0 \to \bullet \to \cO_- \to M^- \to \underline{\circ} \to 0$:
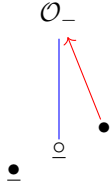
\begin{center}
\begin{sseqdata}[ no axes, struct lines = blue, name = exn0.5k-3, xscale = 0.6, yscale = 0.3, Adams grading, classes = {draw = none} ]
\class["\cO_-"](13,8)
\class["\underline{\bullet}"](12,1)
\class["\underline{\circ}"](13,2)
\class["\bullet"](14,3)
\structline(13,2)(13,8)
\d[red]5(14,3)
\end{sseqdata}
\printpage[ name = exn0.5k-3 ]
\end{center}

For lower values of $k$ we simply pick up additional classes and differentials in the same way as in the previous examples.
\end{example}

\begin{example}
Next consider $\Sigma^{m-\frac{1}{2} \lambda} H \un{M}$. When $k=0$, everything is concentrated in a single slice:
\begin{center}
\begin{sseqdata}[ no axes, struct lines = blue, name = exn-0.5k0, xscale = 0.6, yscale = 0.3, Adams grading, classes = {draw = none} ]
\class["M_-"](-13,-6)
\class["\overline{\circ}"](-12,-7)
\end{sseqdata}
\printpage[ name = exn-0.5k0 ]
\end{center}

When $k=-1$, we pick up a differential and an extension:
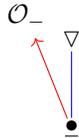
\begin{center}
\begin{sseqdata}[ no axes, struct lines = blue, name = exn-0.5k-1, xscale = 0.6, yscale = 0.3, Adams grading, classes = {draw = none} ]
\class["\cO_-"](-13,-6)
\class["\triangledown"](-12,-7)
\class["\underline{\bullet}"](-12,-11)
\structline(-12,-7)(-12,-11)
\d[red]5(-12,-11)
\end{sseqdata}
\printpage[ name = exn-0.5k-1 ]
\end{center}

When $k=-2$, we pick up one more differential:
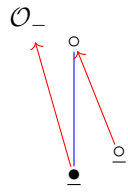
\begin{center}
\begin{sseqdata}[ no axes, struct lines = blue, name = exn-0.5k-2, xscale = 0.6, yscale = 0.3, Adams grading, classes = {draw = none} ]
\class["\cO_-"](-13,-4)
\class["\circ"](-12,-5)
\class["\underline{\bullet}"](-12,-11)
\class["\underline{\circ}"](-11,-10)
\structline(-12,-5)(-12,-11)
\d[red]5(-11,-10)
\d[red]7(-12,-11)
\end{sseqdata}
\printpage[ name = exn-0.5k-2 ]
\end{center}

For lower values of $k$, we simply pick up additional classess and differentials. For $k=-3$ it looks as follows:
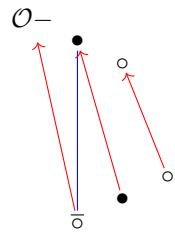
\begin{center}
\begin{sseqdata}[ no axes, struct lines = blue, name = exn-0.5k-3, xscale = 0.6, yscale = 0.3, Adams grading, classes = {draw = none} ]
\class["\cO-"](-13,-2)
\class["\bullet"](-12,-3)
\class["\overline{\circ}"](-12,-11)
\class["\bullet"](-11,-10)
\class["\circ"](-11,-4)
\class["\circ"](-10,-9)
\structline(-12,-3)(-12,-11)
\d[red]5(-10,-9)
\d[red]7(-11,-10)
\d[red]9(-12,-11)
\end{sseqdata}
\printpage[ name = exn-0.5k-3 ]
\end{center}

For $k=1$ we have no differentials or extensions:
\begin{center}
\begin{sseqdata}[ no axes, struct lines = blue, name = exn-0.5k1, xscale = 0.6, yscale = 0.3, Adams grading, classes = {draw = none} ]
\class["M_-"](-13,-6)
\class["\overline{\circ}"](-12,-2)
\end{sseqdata}
\printpage[ name = exn-0.5k1 ]
\end{center}

For $k=2$ we pick up a single extension:
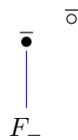
\begin{center}
\begin{sseqdata}[ no axes, struct lines = blue, name = exn-0.5k2, xscale = 0.6, yscale = 0.3, Adams grading, classes = {draw = none} ]
\class["F_-"](-13,-6)
\class["\overline{\circ}"](-12,-1)
\class["\overline{\bullet}"](-13,-2)
\structline(-13,-2)(-13,-6)
\end{sseqdata}
\printpage[ name = exn-0.5k2 ]
\end{center}

For $k=3$ we pick up a differential as well:
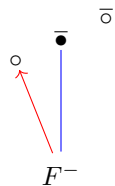
\begin{center}
\begin{sseqdata}[ no axes, struct lines = blue, name = exn-0.5k3, xscale = 0.6, yscale = 0.3, Adams grading, classes = {draw = none} ]
\class["F^-"](-13,-8)
\class["\overline{\circ}"](-12,-1)
\class["\overline{\bullet}"](-13,-2)
\class["\circ"](-14,-3)
\structline(-13,-2)(-13,-8)
\d[red]5(-13,-8)
\end{sseqdata}
\printpage[ name = exn-0.5k3 ]
\end{center}

For larger values of $k$ we simply pick up additional classes and differentials in the same way as in the previous examples.
\end{example}

\begin{example}
For $n \geq 3/2$ a half-integer, the slice spectral sequence looks as follows: For $k \geq 0$ it looks the same as for $n \geq 1$ an integer, but with $F$ replaced by
\begin{center}
\begin{sseqdata}[ no axes, struct lines = blue, name = exn1.5k0, xscale = 0.6, yscale = 0.3, Adams grading, classes = {draw = none} ]
\class["F^-"](1,0)
\class["\bullet"](0,1)
\end{sseqdata}
\printpage[ name = exn1.5k0 ]
\end{center}

For $-2n \leq k \leq -1$, the slice spectral has no differentials or extensions. For example, for $n=3/2$ and $k=-2$ it looks as follows:
\begin{center}
\begin{sseqdata}[ no axes, struct lines = blue, name = exn1.5k-2, xscale = 0.6, yscale = 0.3, Adams grading, classes = {draw = none} ]
\class["F^-"](13,6)
\class["\bullet"](12,7)
\class["\underline{\circ}"](11,2)
\class["\underline{\bullet}"](10,1)
% \class["\underline{\bullet}"](9,0)
\end{sseqdata}
\printpage[ name = exn1.5k-2 ]
\end{center}

For $k = -2n-1$, we pick up an extension $0 \to \circ \to F^- \to \cO^- \to 0$. For example, when $n=3/2$ and $k=-4$ the spectral sequence looks like this:
\begin{center}
\begin{sseqdata}[ no axes, struct lines = blue, name = exn1.5k-4, xscale = 0.6, yscale = 0.3, Adams grading, classes = {draw = none} ]
\class["\cO^-"](13,8)
\class["\circ"](13,4)
\class["\bullet"](12,3)
\class["\underline{\circ}"](11,2)
\class["\underline{\bullet}"](10,1)
\structline(13,4)(13,8)
\end{sseqdata}
\printpage[ name = exn1.5k-4 ]
\end{center}

For $k < -2n-1$ we pick up additional differentials. For example, when $n=3/2$ and $k=-5$ the spectral sequence looks like this:
\begin{center}
\begin{sseqdata}[ no axes, struct lines = blue, name = exn1.5k-5, xscale = 0.6, yscale = 0.3, Adams grading, classes = {draw = none} ]
\class["\cO_-"](13,10)
\class["\bullet"](14,5)
\class["\circ"](13,4)
\class["\bullet"](12,3)
\class["\underline{\circ}"](11,2)
\class["\underline{\bullet}"](10,1)
\structline(13,4)(13,10)
\d[red]5(14,5)
\end{sseqdata}
\printpage[ name = exn1.5k-5 ]
\end{center}
\end{example}

\begin{example}
Finally, for $n \leq -3/2$ a half-integer, the slice spectral sequence looks as follows: For $k \leq 0$ it looks the same as for $n \leq -1$ an integer, but with $\cO$ replaced by

\begin{center}
\begin{sseqdata}[ no axes, struct lines = blue, name = exn-1.5k0, xscale = 0.6, yscale = 0.3, Adams grading, classes = {draw = none} ]
\class["\cO_-"](-1,0)
\class["\circ"](0,-1)
\end{sseqdata}
\printpage[ name = exn-1.5k0 ]
\end{center}

For $-2n \geq k \geq 1$, the slice spectral has no differentials or extensions. For example, for $n=-3/2$ and $k=2$ it looks as follows:
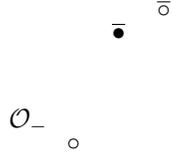
\begin{center}
\begin{sseqdata}[ no axes, struct lines = blue, name = exn-1.5k2, xscale = 0.6, yscale = 0.3, Adams grading, classes = {draw = none} ]
\class["\cO_-"](-13,-6)
\class["\circ"](-12,-7)
\class["\overline{\bullet}"](-11,-2)
\class["\overline{\circ}"](-10,-1)
\end{sseqdata}
\printpage[ name = exn-1.5k2 ]
\end{center}

For $k = -2n+1$, we pick up an extension $0 \to F_- \to \cO_- \to \bullet \to 0$. For example, when $n=-3/2$ and $k=4$ the spectral sequence looks like this:
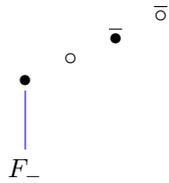
\begin{center}
\begin{sseqdata}[ no axes, struct lines = blue, name = exn-1.5k4, xscale = 0.6, yscale = 0.3, Adams grading, classes = {draw = none} ]
\class["F_-"](-13,-8)
\class["\bullet"](-13,-4)
\class["\circ"](-12,-3)
\class["\overline{\bullet}"](-11,-2)
\class["\overline{\circ}"](-10,-1)
\structline(-13,-4)(-13,-8)
\end{sseqdata}
\printpage[ name = exn-1.5k4 ]
\end{center}

For $k > -2n+1$ we pick up additional differentials. For example, when $n=-3/2$ and $k=5$ the spectral sequence looks like this:
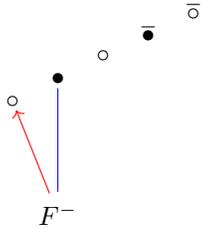
\begin{center}
\begin{sseqdata}[ no axes, struct lines = blue, name = exn-1.5k5, xscale = 0.6, yscale = 0.3, Adams grading, classes = {draw = none} ]
\class["F^-"](-13,-10)
\class["\circ"](-14,-5)
\class["\bullet"](-13,-4)
\class["\circ"](-12,-3)
\class["\overline{\bullet}"](-11,-2)
\class["\overline{\circ}"](-10,-1)
\structline(-13,-4)(-13,-10)
\d[red]5(-13,-10)
\end{sseqdata}
\printpage[ name = exn-1.5k5 ]
\end{center}
\end{example}

This completes the description of the slice spectral sequence for $\Sigma^V H\un{M}$ for all primes $p$, Mackey functors $\un{M}$, and virtual $C_p$-representations $V$.

\begin{example}
For a particular Mackey functor, many of the new Mackey functors constructed from it are isomorphic. For example, consider the constant Mackey functor $\un{\bZ}$. Then
\[
\arraycolsep=6pt\def\arraystretch{2.0}
\begin{array}{cccc}
 R = \un{\bZ} &
 T =\un{\bZ}^* &
 F = \un{\bZ} &
 \cO = \un{\bZ}^* \\
 \circ = 0 &
 \bullet = \wh{\bZ/p} &
 \overline{\circ} = 0 &
 \overline{\bullet} = 0 \\
 \underline{\circ} = 0 &
 \underline{\bullet} = \wh{\bZ/p} &
 \blacktriangle = \wh{\bZ/p} &
 \triangledown = 0
\end{array}
\]
Hence the $12$ Mackey functors from the general case have been reduced to $3$.

To pick just one of the examples from before, consider the spectral sequence for $\Sigma^{m+\lambda} H\un{\bZ}$ when $k=3$. It then simplifies so much that it collapses:
\begin{center}
\begin{sseqdata}[ no axes, struct lines = blue, name = exn1k3Z, xscale = 0.6, yscale = 0.3, Adams grading, classes = {draw = none} ]
\class["{\un{\bZ}}"](13,0)
\class["0"](12,1)
\class["{\wh{\bZ/p}}"](11,2)
\class["0"](10,3)
\class["0"](11,8)
\class["0"](10,7)
\class["0"](9,6)
% \structline(11,2)(11,8)
% \d[red]3(10,3)
% \d[red]5(11,2)
% \d[red]7(12,1)
\end{sseqdata}
\printpage[ name = exn1k3Z ]
\end{center}
\end{example}

\begin{example}
A similar simplification happens for $\un{\bF}_p$. At $p=2$ we invite the reader to verify that this recovers \cite[Theorem 3.18]{GuYa} for $V = n \geq 4$. (That result is of course not the main point of loc.\ cit.)
\end{example}

% \bibliographystyle{plain}
% \bibliography{b.bib}

\end{document}